\documentclass[11pt]{amsart}
\usepackage[utf8]{inputenc}
\usepackage{amssymb, enumerate, todonotes}
\newtheorem{theorem}{Theorem}[section]
\newtheorem*{theorem*}{Theorem}
\newtheorem{lemma}[theorem]{Lemma}
\newtheorem{proposition}[theorem]{Proposition}
\newtheorem{corollary}[theorem]{Corollary}
\theoremstyle{definition}

\theoremstyle{remark}
\newtheorem{remark}[theorem]{Remark}
\numberwithin{equation}{section}

\title[Self-intersections of Laurent polynomials]{Self-intersections of Laurent polynomials and the density of Jordan curves}

\author{Sergei Kalmykov}
\address{School of mathematical sciences, Shanghai Jiao Tong University, 800 Dongchuan RD, Shanghai 200240, China;} \address{ 
Institute of applied mathematics FEB RAS, Vladivostok, 7 Radio str.,  Russia.
} 
\email{kalmykovsergei@sjtu.edu.cn}
\thanks{First author supported by  SJTU start-up grant program WF220407115 and partially by Russian Foundation for Basic Research (grant 18-31-00101)}

\author{Leonid V. Kovalev}
\address{215 Carnegie, Mathematics Department, Syracuse University, Syracuse, NY 13244, USA}
\email{lvkovale@syr.edu}
\thanks{Second author supported by the National Science Foundation grants DMS-1362453 and DMS-1764266.}


\begin{document}

\begin{abstract} We extend Quine's bound on the number of self-intersection of curves with polynomial parameterization to the case of Laurent polynomials. As an application, we show that circle embeddings are dense among all maps from a circle to a plane with respect to an integral norm.
\end{abstract}

\maketitle
\baselineskip6mm

\section{Introduction}

In 1973 Quine~\cite{Quine73} proved that, with few exceptions, the restriction of a complex polynomial of degree $n$ to the unit circle $\mathbb T$ is a closed curve with at most $(n-1)^2$ self-intersections, and this upper bound is best possible. The exceptional case is the polynomial being of the form $p(z)=q(z^j)$ where $q$ is a polynomial and $j>1$. 

In the context of continuous circle maps $f\colon \mathbb T\to \mathbb C$ it is natural to consider Laurent polynomials $p(z)= \sum_{k=m}^n c_k z^k$, which can approximate $f$ uniformly. Our main result (Theorem~\ref{self-intersection-thm}) asserts, in part, that the closed curve $p_{|\mathbb T}$ has at most $(n-1)(n-m)$ self-intersections when $-n < m < 0$, unless $p$ is of the form $q(z^j)$ where $q$ is a Laurent polynomial and $j\ne -1, 1$. This estimate is sharp when $\gcd(n, m)=1$, as is shown in Section~\ref{lower-bounds-sec}. It also matches Quine's bound $(n-1)^2$ which corresponds to $m=1$. 

As a consequence of the finiteness of self-intersections, we obtain the density of circle embeddings in $L^p$ norms for finite $p$.   

\begin{theorem*}[Theorem~\ref{approximation-thm}]
For $p\in [1, \infty)$, every function $f\in L^p(\mathbb T; \mathbb C)$ can be approximated in the $L^p$ norm by orientation-preserving 
$C^\infty$-smooth embeddings of $\mathbb T$ into $\mathbb C$. 
\end{theorem*}

When $p=2$, it follows that one can obtain no quantitative estimates for the Fourier coefficients $\hat f$ based on the fact that $f$ is an embedding, even if its orientation is known. Such estimates are available under additional geometric conditions such as convexity or starlikeness of $f(\mathbb T)$: e.g., the Rad\'o-Kneser-Choquet theorem~\cite[p. 29]{Duren} implies that $|\hat f(1)|>|\hat f(-1)|$ for positively oriented convex curves. The relation between $\hat f$ and the shape of $f(\mathbb T)$ was considered in~\cite{Hall, KovalevYang}. 

\section{Self-intersections of Laurent polynomials}

Consider a Laurent polynomial
\begin{equation}\label{Lpoly}
p(z) = \sum_{k=m}^n a_k z^k,\quad z\in \mathbb C\setminus \{0\}, 
\end{equation}
where $m, n\in \mathbb Z$, $a_m\ne 0$, and $a_n\ne 0$. On the unit circle $\mathbb T$ this can be written as a trigonometric polynomial, 
\begin{equation}\label{Lpolytrig}
p(e^{i\theta}) = \sum_{k=m}^n a_k e^{ik\theta}, \quad \theta\in \mathbb R. 
\end{equation}
We are interested in the self-intersections of the closed parametric curve $p(\mathbb T)=\{p(e^{i\theta})\colon 0\le \theta\le 2\pi\}$.
By definition, a \textit{self-intersection of $p$ on $\mathbb T$} is a two-point subset  $\{z_1, z_2\}\subset \mathbb T$ where $z_1\ne z_2$ and $p(z_1)=p(z_2)$. For example, the image of $\mathbb T$ under $p(z) = z^2 + z^{-1}$ passes through $0$ three times, which counts as three self-intersections, namely $\{e^{\pi i/3}, -1\}$, $\{e^{-\pi i/3}, -1\}$, and $\{e^{\pi i/3}, e^{-\pi i/3}\}$. To motivate this way of counting, observe that the image of $\mathbb T$ under a perturbed function $z^2 + cz^{-1}$ with $c$ close to $1$ has three distinct self-intersections near $0$.

Replacing $\theta$ by $-\theta$ if necessary, we make sure that $n \ge |m|$. Also, since the constant term does not affect self-intersections, we may assume $m\ne 0$. Thus, the case of algebraic polynomials considered by Quine~\cite{Quine73} corresponds to $m=1$. It should be noted that Quine considers the \textit{vertices} of $p$, which are the values attained more than once. The number of vertices may be smaller than the number of self-intersections, but Quine's argument applies to both.  The main result of this paper is the following theorem. 

\begin{theorem}\label{self-intersection-thm} If $-n \le m < n$ and $m\ne0$, the number of self-intersections of the Laurent polynomial~\eqref{Lpoly} on $\mathbb T$ is at most 
\begin{equation}\label{upper-bound-thm}
\begin{cases}
(n-1) \left(n - \frac{m+1}{2}\right),\quad & 1\le m < n  \\
(n - 1) (n - m),\quad & -n < m \le -1 \\
(n - 1) (2n - 1), \quad & m = -n
\end{cases}
\end{equation}
with the following exceptions: (a) $p$ can be written as $q(z^j)$ for some Laurent polynomial $q$ and some integer $j\ne -1, 1$; (b) $n=-m$ and $|a_n| =  |a_m|$. 
\end{theorem}

\begin{remark}
If $p=q(z^j)$ with $j\ne -1, 1$, the polynomial $p$ traces a closed curve more than once, thus creating uncountably many self-intersections. If $n=-m$ and $|a_n|=|a_m|$, the number of self-intersections may also be infinite: consider $p(z) = q(z+1/z)$ where $q$ is an algebraic polynomial of degree $n$. This polynomial has self-intersections $p(z)=p(1/z)$, for all $z\in \mathbb T\setminus \{-1, 1\}$. 
\end{remark}

The sharpness of Theorem~\ref{self-intersection-thm} is discussed in Section~\ref{lower-bounds-sec}. Its proof requires preliminary lemmas involving Chebyshev polynomials and resultants. 

Let $U_n$, $n\in \mathbb N$, be the Chebyshev polynomial of second kind of degree $n$. Recall that 
\begin{equation}\label{Chebyshev-U-cos}
U_n(\cos \theta) = \frac{\sin((n+1)\theta)}{\sin \theta}.    
\end{equation}
By convention, $U_{-1}\equiv 0$ and $U_{-n-1} = -U_{n-1}$ for $n\in \mathbb N$; both of these formulas are consistent with~\eqref{Chebyshev-U-cos}.

\begin{lemma}\label{g-lemma} Consider a Laurent polynomial~\eqref{Lpoly} with $-n<m\le n$, $m\ne 0$, $a_n\ne 0$, and $a_m\ne 0$. Let 
\begin{equation}\label{g-note}
g(t, z) = \sum_{k=m}^{n} a_k U_{k-1}(t) z^{k-m},
\end{equation}
\begin{equation}\label{gs-note}
g^*(t, z) = z^{n-m} \overline{g(\bar t, 1/\bar z)} = \sum_{k=m}^{n} \overline{a_k} U_{k-1}(t) z^{n-k}.
\end{equation}
Then, with $t=\cos \theta$, we have
\begin{equation}\label{g-notation}
g(t, z) = z^{-m} \frac{p(e^{i\theta} z) - p(e^{-i\theta} z)}{e^{i\theta} - e^{-i\theta}}.
\end{equation}
Also, $g$ is a polynomial in $t, z$ of total degree $2n-m-1$, and $g^*$ is a polynomial in $t, z$ of total degree 
$n - m + |m| - 1$.  Finally, if $g$ and $g^*$ are considered as elements of $\mathbb C[t][z]$, their resultant is a polynomial of degree $2(n-1)(n-m)$ in $t$.  
\end{lemma}

\begin{proof} The property~\eqref{g-notation} follows by expanding $p(e^{\pm i\theta} z)$ into a sum and observing that
\[
\frac{e^{ik\theta} - e^{-ik\theta}}{e^{i\theta} - e^{-i\theta}} = \frac{\sin k\theta}{\sin \theta} = U_{k-1}(t).
\]
The  monomial of highest degree in $g$ 
comes from multiplying the leading term of $U_{n-1}(t)$ by $a_nz^{n-m}$. This leading term is $(2t)^{n-1}$ (see, e.g. \cite[2.1.E.10(g), p.~37]{BorweinErdelyi}), and therefore the leading monomial in $g$ is $a_n (2t)^{n-1}z^{n-m}$.  

Concerning $g^*$, note that the total degree of $U_{k-1}(t) z^{n-k}$ is $|k|-1 + (n-k)$, which is strictly decreasing for negative $k$ and constant for positive $k$. Thus, if $m<0$, the monomial of highest degree  in $g^*$ is $-\overline{a_m}(2t)^{-m-1} z^{n-m}$, which has degree $n - 2m -1$. If $m>0$, then the highest degree is $n-1$, which is is achieved by multiple monomials. Their sum is
\begin{equation}\label{max-deg-gstar} 
 \sum_{k=m}^n \overline{a_k} (2t)^{k-1}z^{n-k}. 
\end{equation}
There is no cancellation between monomials in~\eqref{max-deg-gstar}. 

For the computation of the resultant of $g$ and $g^*$, write down their Sylvester matrix (see e.g. \cite[(1.12), Chapter 12, p. 400]{Gelfand} or \cite[Chapter 1, p.~24]{Walker}) as
\[
\begin{pmatrix}
a_n U_{n-1} & \cdots  & a_m U_{m-1} & \\
\ddots & & \ddots & \\
& a_n U_{n-1} & \cdots  & a_m U_{m-1} \\ 
\overline{a_m} U_{m-1} & \cdots  & \overline{a_n} U_{n-1} & \\
\ddots & & \ddots & \\
& \overline{a_m} U_{m-1} & \cdots  & \overline{a_n} U_{n-1}  
\end{pmatrix}
\]
which is a matrix of size $2(n-m)$ where the diagonal elements are of degree $n-1$ in $t$, while off-diagonal elements are of degree less than $n-1$. It follows that the determinant of this matrix is a polynomial of degree $2(n-1)(n-m)$ in $t$.
\end{proof}

\begin{lemma}\label{self-intersections} Let $p$, $g$, $g^*$ be as in Lemma~\ref{g-lemma}. Given a self-intersection of $p_{|\mathbb T}$, write it in the form $\{ e^{i\theta}z, e^{-i\theta}z \}$ where $z\in \mathbb T$ and $e^{i\theta}\in \mathbb T\setminus \{-1, 1\}$. Let $t=\cos \theta$. Then $g(t, z)=g^*(t, z) = g(-t, -z) = g^*(-t, -z)=0$, i.e., the algebraic curves $g=0$ and $g^*=0$ intersect at the points $(t, z)$ and $(-t, -z)$. Different self-intersections correspond to different pairs $\{(t, z), (-t, -z)\}$.  
\end{lemma}

\begin{proof} The identity~\eqref{g-notation} implies $g(t, z)=0$. Since $z=1/\bar z$, we also have $g^*(t, z)=0$ from~\eqref{gs-note}. Since the self-intersection 
$\{ e^{i\theta}z, e^{-i\theta}z \}$ can also be written as $\{ e^{i(\pi-\theta)}(-z), e^{-i(\pi-\theta)}(-z) \}$, it follows that $g$ and $g^*$ vanish at $(-t, -z)$ as well. Finally, the pair $(t, z)$ determines $(\theta, z)$ up to replacing $\theta$ with $\pm \theta + 2\pi k$, $k\in \mathbb Z$, which does not change the 
self-intersection set $\{ e^{i\theta}z, e^{-i\theta}z \}$.
\end{proof}

\begin{proof}[Proof of Theorem~\ref{self-intersection-thm}]
We begin with the case $n=-m$. For $z\in \mathbb T$ the Laurent polynomial $p$ agrees with the harmonic polynomial
\begin{equation}\label{harmonic-p}
    p_h(z) = \sum_{k=1}^n (a_k z^k + a_{-k} {\bar z}^{k}).
\end{equation}
Let $\psi(w) = w + c\overline{w}$, where $c=-a_{-n}/\overline{a_n}$. This is an invertible $\mathbb R$-linear transformation of the plane, with the inverse $\psi^{-1}(\zeta) =  (\zeta-c\overline{\zeta})/(1-|c|^2)$. 
We have 
\[
\psi\circ p_h(z) = 
\sum_{k=1}^n ((a_k + c \overline{a_{-k}}) z^k + 
(a_{-k} +  c \overline{a_k}) {\bar z}^{k}),
\]
where the coefficient of $\bar z^n$ vanishes by the choice of $c$. Returning to the Laurent polynomial form, we have for $z\in \mathbb T$, 
\[
\psi\circ p(z) = 
\sum_{k=1}^n (a_k + c \overline{a_{-k}}) z^k + 
\sum_{k=1}^{n-1} (a_{-k} +  c \overline{a_k}) {z}^{-k}. 
\]
If $\psi\circ p_{|\mathbb T}$ depends only on $z^j$ for some $j\in \mathbb Z\setminus \{-1, 1\}$, then by applying the inverse transformation $\psi^{-1}$ we conclude that the original polynomial $p$ had the same property, i.e., exceptional case (a) holds. Apart from this exceptional case, we can apply Theorem~\ref{self-intersection-thm} to $\psi\circ p$, with $m>-n$. The bound provided by~\eqref{upper-bound-thm} is largest when $m=1-n$, when it is equal to $(n-1)(2n-1)$. This completes the case $n=-m$. 

From now on, $-n<m < n$. Let $g$ and $g^*$ be as in Lemma~\ref{g-lemma}. The polynomials $g$ and $g^*$ are relatively prime in $\mathbb C[t][z]$, for otherwise their resultant would be identically zero, contradicting Lemma~\ref{g-lemma}. We also want to show they are relatively prime in $\mathbb C[t, z]$. If not, there is a nonconstant polynomial $h\in \mathbb C[t]$ that divides both $g$ and $g^*$. Let $t_0$ be a zero of $h$. Then $g(t_0, z)=0$ for all $z$, which in view of~\eqref{g-note} implies $U_{n-1}(t_0)=0$. The definition of the Chebyshev polynomial $U_{n-1}$ implies that $t_0 = \cos(\pi k/n)$ for some integer $1\le k \le n-1$. By virtue of~\eqref{g-notation} we have $p(e^{2k \pi i/n}z) = p(z)$ for all $z$. Comparing the coefficients of these Laurent polynomials, we conclude that $p$ can be written in the form $p(z) = q(z^j)$ where $j$ is such that $e^{2k \pi i/n}$ is a primitive $j$th root of unity, and $q$ is a Laurent polynomial. This is the exceptional case (a) of the theorem.   

Thus, $g$ and $g^*$ are relatively prime in $\mathbb C[t, z]$. By Bezout's theorem ~\cite[Chapter 3, Theorem~3.1, p.~59]{Walker}  they have at most $\deg g\, \deg g^*$ common zeros. By Lemma~\ref{self-intersections}, the number of self-intersections of $p_{|\mathbb T}$ is at most $\frac12\deg g\, \deg g^*$. This proves the case $m\ge 1$ of~\eqref{upper-bound-thm}. The case $m\le -1$ requires additional consideration of the intersection between $g=0$ and $g^*=0$ at infinity, similar to the proof of Theorem~3 in~\cite{Quine76}. 

Recalling~\eqref{g-note} and~\eqref{gs-note}, we can write the polynomials $g$ and $g^*$ 
in terms of homogeneous coordinates $(t, z, w)$ as follows:
\begin{equation}\label{G-homog}
G(t, z, w) = w^{2n-m-1} g(t/w, z/w)
= \sum_{k=m}^n a_k U_{k-1}(t/w) z^{k-m} w^{2n-k-1}
\end{equation}
and
\begin{equation}\label{Gs-homog}
G^*(t, z, w) = w^{n-2m-1} g^*(t/w, z/w)
= \sum_{k=m}^n a_k U_{k-1}(t/w) z^{n-k} w^{k-2m-1}.
\end{equation}

Since $w^{|k|-1}U_{k-1}(t/w)$ is a polynomial, the index-$k$ term in~\eqref{G-homog} is divisible by $z^{k-m}w^{2n-k-|k|}$ which is a monomial of degree $2n-m-|k| \ge n-m$. Thus, $G$ has a zero of order $n-m$ at the point $(t, z, w) = (1, 0, 0)$ of the projective space $\mathbb {CP}^2$. Similarly, 
the index-$k$ term of~\eqref{Gs-homog} is divisible by the monomial $z^{n-k}w^{k - 2m - |k|}$ of degree $n-2m-|k| \ge -2m$. Thus, $G^*$ has a zero of order $-2m$ at the point $(1, 0, 0)$ of $\mathbb {CP}^2$. By Theorem~5.10 in~\cite[p.~114]{Walker}, the curves $G=0$ and $G^*=0$ have an intersection of multiplicity at least $(-2m)(n-m)$ at $(1, 0, 0)$. 

Since the index-$k$ term in the sum~\eqref{G-homog} has degree $2n-k-1\ge n-1$ in $t$ and $w$ jointly, it follows that $G$ has a zero of order $n-1$ at $(0, 1, 0)$. Also, the index $k$ term in~\eqref{Gs-homog} has degree $k - 2m - 1\ge -m - 1$ in $t$ and $w$ jointly, which implies that $G$ has a zero of order $n-1$ at $(0, 1, 0)$. (As usual, a zero of order $0$ is not a zero at all.) This results in the intersection multiplicity at least $(n-1)(-m-1)$ at $(0, 1, 0)$. 

Subtracting the intersections at $(1, 0, 0)$ and $(0, 1, 0)$ from the total number $\deg g \deg g^*$ given by Bezout's theorem, we conclude that the curves $g=0$ and $g^*=0$ have at most 
\[
(2n-m-1)(n-2m-1) + 2m(n-m) + (n-1)(m+1) = 2(n-1)(n-m)
\]
intersections in the affine plane $\mathbb C^2$. By Lemma~\ref{self-intersections}, the number of  self-intersections of $p_{|\mathbb T}$ is bounded by $(n-1)(n-m)$, in agreement with~\eqref{upper-bound-thm}. This completes the proof of Theorem~\ref{self-intersection-thm}.
\end{proof}

\section{Lower bound on the number of self-intersections}\label{lower-bounds-sec}

The case of algebraic polynomials, considered by Quine~\cite{Quine73}, corresponds to $m=1$ in Theorem~\ref{self-intersection-thm}, when the estimate on the number of self-intersections is $(n-1)^2$. This bound is attained by $z^n+\epsilon z$ for small $\epsilon$, as shown in~\cite{Quine73}. 

The following proposition implies that the bound provided by Theorem~\ref{self-intersection-thm} is also sharp when $m$ is negative and coprime to $n$.

\begin{proposition} Suppose $n, m\in \mathbb Z$, $n > |m|\ge 1$, and $\gcd(n, m)=1$. Then for sufficiently small $\epsilon>0$ the Laurent polynomial $p(z) = z^n + \epsilon z^m$ has $(n-1)(n-m)$ self-intersections on $\mathbb T$. 
\end{proposition}

\begin{proof} The polynomial $g$ from~\eqref{g-note} takes the form 
\begin{equation}\label{g-mn1}
  g(t, z) = U_{n-1}(t) z^n + \epsilon U_{m-1}(t)z^m = \left(\frac{\sin n\theta}{\sin m\theta} z^{n-m} + \epsilon\right)
  \frac{\sin m\theta}{\sin \theta} z^m  
\end{equation}
where $t=\cos \theta$. Note that $U_{n-1}$ and $U_{m-1}$ have no common zeros because $\gcd(n, m)=1$. Therefore, any solution of $g(t, z)=0$ with $|z|=1$ and $0<t<1$ arises from
\begin{equation}\label{g-mn2}
\frac{\sin n\theta}{\sin m\theta} = \pm \epsilon, 
\quad 0 < \theta < \frac{\pi}{2}. 
\end{equation}
The zeros of the left-hand side of~\eqref{g-mn2} on $[0, \pi/2]$ are $\pi k/n$ for $1\le k\le \lfloor n/2\rfloor $. It follows that for small enough $\epsilon$,~\eqref{g-mn2} holds at $n-1$ points of $(0, \pi/2)$. Indeed, there are two such points near $\pi k/n$ with $1\le k<\lfloor n/2\rfloor$, and one such point next to $\pi/2$ (only if $n$ is even). This adds up to $2(n/2-1)+1 = n-1$ when $n$ is even, and $2(n-1)/2 = n-1$ when $n$ is odd. 

Thus, we have $n-1$ values of $t\in (0, 1)$ for which $|U_{n-1}(t)| = \epsilon|U_{m-1}(t)|$, and for each of them there are $(n-m)$ values os $z$ (roots of either $1$ or $-1$) such that \eqref{g-mn1} turns into $0$. Each such pair $(t, z)$ produces a self-intersection of $p_{|\mathbb T}$ by virtue of~\eqref{g-notation}, and all these self-intersections are distinct by Lemma~\ref{self-intersections}. In conclusion, there are $(n-1)(n-m)$  self-intersections of $p_{|\mathbb T}$.
\end{proof}

We do not know whether Theorem~\ref{self-intersection-thm} is sharp when $m$ and $n$ are not coprime, or when $m>1$.  
 
\section{Approximating closed curves by Jordan curves}\label{approximation-sec}

Let $\mathcal E(\mathbb T;\mathbb C)$ be the set of all circle embeddings, i.e., continuous injective maps of $\mathbb T$ into $\mathbb C$. It is well known that continuous maps are dense in $L^p(\mathbb T; \mathbb C)$ for $1\le p<\infty$. In this section we prove that $\mathcal E(\mathbb T;\mathbb C)$ is dense as well. As a corollary, it follows that the Fourier coefficients $\hat f$ of a circle embedding $f\colon \mathbb T\to\mathbb C$ can be arbitrarily close to any element of $\ell^2(\mathbb Z)$.  
 
Note that the real-variable analog of this result is false: continuous injective maps $f\colon [0, 1]\to \mathbb R$ are not dense in $L^p([0, 1])$ for any $p$, as their closure is the set of monotone functions. Also, $\mathcal E(\mathbb T;\mathbb C)$ is not dense in the space of continuous maps $C^0(\mathbb T; \mathbb C)$ with the uniform norm, e.g., a continuous map of $\mathbb T$ onto a ``figure eight'' curve cannot be uniformly approximated by injective maps.  
 
\begin{theorem}\label{approximation-thm}
For $p\in [1, \infty)$, every function $f\in L^p(\mathbb T; \mathbb C)$ can be approximated in the $L^p$ norm by orientation-preserving 
$C^\infty$-smooth embeddings of $\mathbb T$ into $\mathbb C$. 
\end{theorem}

\begin{proof}
By the Stone-Weierstrass theorem, the Laurent polynomials $q(z)=\sum_{k=m}^n a_k z^k$ are dense in $C^0(\mathbb T;\mathbb T)$, hence dense in $L^p$. By a slight perturbation we can ensure that $q$ does not fall into either of the exceptional cases of Theorem~\ref{self-intersection-thm} and therefore $q_{|\mathbb T}$ has a finite number of self-intersections. Consequently, there is a finite subset $F\subset \mathbb T$ such that $q$ is injective on $\mathbb T\setminus F$. 

After removing small disjoint neighborhoods of the elements of $F$ from $\mathbb T$, we obtain a finite set of disjoint arcs $\gamma_j\subset \mathbb T$, $j=1,\dots,N$ whose images under $q$ are disjoint smooth simple arcs $\Gamma_j=q(\gamma_j)$, $j=1,\dots, N$. Recall that a simple arc (a homeomorphic image of a line segment) does not separate the plane~\cite[Theorem V.10.1]{Newman}. By Janiszewski's theorem~\cite[Theorem V.9.1.2]{Newman}, the set $\Omega = \mathbb C\setminus \bigcup_{j=1}^N \Gamma_j$ is connected.  

The arcs $\Gamma_j$ have orientation induced by the positive (counterclockwise) orientation of $\mathbb T$. Since the complement of $\Omega$ consists of smooth arcs, every boundary point of $\Omega$ is accessible from the domain by a smooth curve. In particular, we can join the endpoint of $\Gamma_1$ to the beginning of $\Gamma_2$ by a smooth curve that stays within $\Omega$. This replaces $\Gamma_1$ and $\Gamma_2$ by one simple arc, which we can make smooth as well.

Continue the above process until only one smooth oriented arc is left. We have two topologically different ways to join its ends, creating either a positively oriented simple closed curve, or a negatively oriented one. Up to a global homeomorphism, this choice amounts by completing the oriented segment $[-1, 1]$ either by the upper semicircle with counterclockwise orientation, or by the lower semicircle with clockwise orientation. We choose the closed curve to be positively oriented.

It remains to consider the impact of the above modifications on the $L^p$ norm of the parameterized curve $q\colon \mathbb T\to \mathbb C$. To do this, from the beginning we pick a large $R$ such that $|q|<R$ on $\mathbb T$, and perform the replacements so that the connecting curves remain within the open disk $\{w\colon |w|<R\}$. Then the  $L^p$ distance between the original and modified curves is controlled by the linear measure of the set on which $q$ is modified, and this measure can be made arbitrarily small. 
\end{proof}

The Fourier coefficients of an integrable function $f\colon \mathbb T\to\mathbb C$ are given by
\[
\hat f(n) = \frac1{2\pi}\int_{0}^{2\pi} f(e^{i\theta})e^{-in\theta}\,d\theta.
\]
Theorem~\ref{approximation-thm} and Parseval's theorem imply the following result.

\begin{corollary}\label{l2-dense} For any sequence $c\in \ell^2(\mathbb Z)$ and any $\epsilon>0$ there exists an orientation-preserving circle embedding
$f\colon \mathbb T\to \mathbb C$ such that $\|c-\hat f\|_{2}<\epsilon$.
\end{corollary}

Such density no longer holds in some weighted $\ell^2$ norms. For example, $\sum_{n\in \mathbb Z} n|\hat f(n)|^2 > 0$ for every orientation-preserving circle embedding, as this quantity is proportional to the area enclosed by $f(\mathbb T)$.

\bibliographystyle{amsplain} 
\bibliography{references.bib}

\providecommand{\bysame}{\leavevmode\hbox to3em{\hrulefill}\thinspace}
\providecommand{\MR}{\relax\ifhmode\unskip\space\fi MR }
\providecommand{\MRhref}[2]{%
  \href{http://www.ams.org/mathscinet-getitem?mr=#1}{#2}
}
\providecommand{\href}[2]{#2}
\begin{thebibliography}{1}

\bibitem{BorweinErdelyi}
Peter Borwein and Tam\'{a}s Erd\'{e}lyi, \emph{Polynomials and polynomial
  inequalities}, Graduate Texts in Mathematics, vol. 161, Springer-Verlag, New
  York, 1995. \MR{1367960}

\bibitem{Duren}
Peter Duren, \emph{Harmonic mappings in the plane}, Cambridge Tracts in
  Mathematics, vol. 156, Cambridge University Press, Cambridge, 2004.
  \MR{2048384}

\bibitem{Gelfand}
I.~M. Gelfand, M.~M. Kapranov, and A.~V. Zelevinsky, \emph{Discriminants,
  resultants and multidimensional determinants}, Modern Birkh\"{a}user
  Classics, Birkh\"{a}user Boston, Inc., Boston, MA, 2008, Reprint of the 1994
  edition. \MR{2394437}

\bibitem{Hall}
R.~R. Hall, \emph{On an inequality of {E}. {H}einz}, J. Analyse Math.
  \textbf{42} (1982/83), 185--198. \MR{729409}

\bibitem{KovalevYang}
Leonid~V. {Kovalev} and Xuerui {Yang}, \emph{{Fourier series of circle
  embeddings}}, arXiv e-prints (2018), arXiv:1808.04817.

\bibitem{Newman}
M.~H.~A. Newman, \emph{Elements of the topology of plane sets of points},
  Second edition, reprinted, Cambridge University Press, New York, 1961.
  \MR{0132534}

\bibitem{Quine73}
J.~R. Quine, \emph{On the self-intersections of the image of the unit circle
  under a polynomial mapping}, Proc. Amer. Math. Soc. \textbf{39} (1973),
  135--140. \MR{0313485}

\bibitem{Quine76}
\bysame, \emph{Some consequences of the algebraic nature of
  {$p(e^{i\theta})$}}, Trans. Amer. Math. Soc. \textbf{224} (1976), no.~2,
  437--442 (1977). \MR{0419743}

\bibitem{Walker}
Robert~J. Walker, \emph{Algebraic curves}, Springer-Verlag, New
  York-Heidelberg, 1978, Reprint of the 1950 edition. \MR{513824}

\end{thebibliography}

\end{document}